\documentclass[12pt]{amsart}
\usepackage{amsmath,amssymb,amsfonts,amsthm,amsopn}
\usepackage{graphics}

\def\epsilon{\varepsilon}
\def\phi{\varphi}

\newtheorem{theorem}{Theorem}[section]
\newtheorem{lemma}[theorem]{Lemma}
\newtheorem{corollary}[theorem]{Corollary}
\newtheorem{proposition}[theorem]{Proposition}
\newtheorem{definition}[theorem]{Definition}

\newtheorem{remark}[theorem]{Remark}

\newcommand{\beqa}{\begin{eqnarray*}}
\newcommand{\eeqa}{\end{eqnarray*}}

\newcommand{\field}[1]{\mathbb{#1}}
\newcommand{\bR}{\field{R}}        %  real numbers
\newcommand{\bN}{\field{N}}        %  natural numbers
\newcommand{\bZ}{\field{Z}}        %  whole numbers
\newcommand{\bC}{\field{C}}        %  complex numbers
\def\cS{\mathcal{S}}

\def\a{\aleph}

\def\rd{\bR^d}

\def\rdd{{\bR^{2d}}}

\def\cQ{\mathcal{Q}}

\def\R{\right)}

\def\<{\left<}
\def\>{\right>}

\def\mv1{M_v^1}

\def\mn{(m,n)}
\def\mn'{(m',n')}

\def\a{\alpha}

\def\R{\mathbb{R}}
\def\Ren{\mathbb{R}^d}

\def\Sn2{S_{2}(L^{2}(\Ren))}
\def\S1{S_{1}(L^{2}(\Ren))}
\def\sig00{\sigma_{0,0}}
\begin{document}
\begin{abstract}
Consider a non-negative self-adjoint operator $H$ in $L^2(\rd)$. We suppose that its heat operator $e^{-tH}$ satisfies an off-diagonal algebraic decay estimate, for some exponents $p_0\in[0,2)$. Then we prove sharp $L^p\to L^p$ frequency truncated estimates for the Schr\"odinger group $e^{itH}$ for $p\in[p_0,p'_0]$. \par
In particular, our results apply to every operator of the form $H=(i\nabla+A)^2+V$, with a magnetic potential $A\in L^2_{loc}(\rd,\rd)$ and an electric potential $V$ whose positive and negative parts are in the local Kato class and in the Kato class, respectively.
\end{abstract}

\title{Sharp $L^p$ estimates for Schr\"odinger groups}

\author{Piero D'Ancona \and Fabio Nicola}
\address{Sapienza -- Universit\`a di Roma, Dipartimento di Matematica,
Piazzale A. Moro 2, I-00185 Roma, Italy}
\address{Dipartimento di Scienze Matematiche, Politecnico di
Torino, Corso Duca degli
Abruzzi 24, 10129 Torino,
Italy}
%\curraddr{}
\email{dancona@mat.uniroma1.it}
\email{fabio.nicola@polito.it}
%\address{Dipartimento di Matematica, Politecnico di Torino, corso Duca degli Abruzzi 24, 10129 Torino, Italy}

\email{}
%\email{fabio.nicola@polito.it}
\thanks{}

%\date{}
\maketitle
\section{Introduction}
It is well-known that the Schr\"odinger group $e^{it\Delta}$ is bounded on $L^p(\rd)$ only for $p=2$ (if $t\not=0$). However, frequency truncated estimates still hold, which can for instance be phrased in the form
\begin{equation}\label{zero}
\|e^{it\Delta}\phi(2^{-k}D) f\|_{L^p}\lesssim (1+2^{2k}|t|)^{s}\|f\|_{L^p},\quad k\in\mathbb{Z},\ t\in\R,
\end{equation}
for $1\leq p\leq\infty$, $s=d\big|\frac{1}{2}-\frac{1}{p}\big|$, where $\phi\in C^\infty_c (\R^d)$ is a cut-off function. This result follows at once from the stationary phase theorem and is  sharp, both for the growth in $t$ and for the loss of derivatives, i.e.\ the factor $2^{2k}$ in the right-hand side (see \cite{BTW,lanconelli,sjo70}).\par
In this paper we show a far-reaching generalization of this result, to every self-adjoint non-negative operator $H$ in $L^2(\rd)$, whose heat operator $e^{-tH}$ satisfies a mild smoothness effect and a mild off-diagonal decay. \par
%The relevant Besov spaces are those adapted to $H$, i.e. \[ {B}^{s}_{p,q}(\sqrt{H})=\{f\in L^2(\rd):\, \|2^{js}\|\psi_j(H) f\|_{L^p}\|_{\ell^q_j(\bN)}<\infty\},\] where $\psi_0\in C^\infty_0(\R)$, $\psi=1$ in a neighborhood of the origin, $\psi_j(x)=\psi_0(2^{-2j}x)-\psi (2^{-2j+2}x)$, $j=0,1,2,\ldots$.\par
%The desired generalization is therefore a matter of continuity in $L^p$ for a spectral multiplier. However, here we are also interested in the optimal growth in $t$, which makes the problem considerably more delicate. \par
Concerning strong $(p,p)$ estimates of spectral multipliers, recently much attention has been devoted to minimal assumptions on $H$. A condition which is nowadays common in the literature, after \cite{DOS02,heb}, and that already covers a lot of interesting operators, is a pointwise Gaussian estimate for the heat kernel $p_t(x,y)$ of $e^{-tH}$, namely 
\begin{equation}\label{cal}
|p_t(x,y)|\lesssim t^{-d/m} \exp\Big(-b\big(t^{-1/m}|x-y|\big)^{\frac{m}{m-1}}\Big),\quad \ t>0,\ x,y\in\rd,
\end{equation}
for some $b>0$, $m>1$. For example, every Schr\"odinger operator with an electromagnetic potential, under very natural assumptions, satisfies such estimate with $m=2$ \cite{CD,dancona}. As another example, the operator $H=(-\Delta)^{k}$, with $k\geq 1$ integer, satisfies \eqref{cal} with $m=2k$.\par
Recently, motivated by Schr\"odinger operators with bad potentials \cite{sv94} or higher order operators with measurable coefficients \cite{dav95}, the assumptions on $H$ were further weakened in the form of the so-called Generalized Gaussian Estimates, namely
\begin{equation}\label{gge}
\|\mathbf{1}_{B(x,t^{1/m})} e^{-tH} \mathbf{1}_{B(y,t^{1/m})}\|_{L^{p_0}\to L^{p'_0}}\lesssim t^{-\frac{d}{m}\big(\frac{1}{p_0}-\frac{1}{p'_0}\big)}\exp\Big(-b\big(t^{-1/m}|x-y|\big)^{\frac{m}{m-1}}\Big),
\end{equation}
for some $p_0\in[1,2)$ and every $t>0,\ x,y\in\rd$, where $b>0$ and $m>1$; see \cite{Blu03,Blu07,KU12,KU12bis}. When $p_0=1$, \eqref{gge} is in fact equivalent to \eqref{cal} \cite{BK02}.  \par
In the present paper we will consider even weaker estimates, allowing off-diagonal algebraic decay. \par
For every $j\in{\bZ}$, let $\cQ_j$ be the collection of all dyadic cubes in $\rd$ with sidelength $2^{-j}$.  \par\bigskip
{\bf Assumption (H)} {\it Assume that $H$ is a self-adjoint non-negative operator in $L^2(\R^d)$, whose heat operator satisfies the following estimates. There exist $p_0\in[1,2)$, $m>0$ such that for every $t>0$ and $j\in\bZ$, with $2^{-j}\leq t^{1/m}<2^{-j+1}$, we have
\begin{equation}\label{uno}
\sup_{Q'\in\cQ_j}\sum_{Q\in\cQ_j} \|\mathbf{1}_Q e^{-tH} \mathbf{1}_{Q'}\|_{L^{p_0}\to L^{p'_0}}\lesssim 2^{jd\big(\frac{1}{p_0}-\frac{1}{p'_0}\big)}
\end{equation}
and
\begin{equation}\label{due}
\sup_{Q'\in\cQ_j}\sum_{Q\in\cQ_j} (1+2^j{\rm dist}\,(Q,Q'))^{N}\|\mathbf{1}_Q e^{-tH} \mathbf{1}_{Q'}\|_{L^{2}\to L^2}\lesssim 1,\qquad N=\lfloor d/2\rfloor+1.
\end{equation}
} \par\bigskip
Of course, if the pointwise bound \eqref{cal} holds, then \eqref{uno} is satisfied with $p_0=1$, as well as \eqref{due}. Also, \eqref{gge} implies \eqref{uno} and \eqref{due}. Notice however that there are operators satisfying \eqref{uno} and \eqref{due} but not \eqref{gge}. For example, for the fractional Laplacian $H=(-\Delta)^\alpha$, $\alpha>0$, \eqref{uno} is satisfied with $p_0=1$ for every $\alpha>0$, whereas \eqref{due} holds for $2\alpha> \lfloor d/2\rfloor+1$ (both with $m=2\alpha$); see Section 5 below. \par
Now, we can state our main result.
\begin{theorem}\label{mainteo}
Assume the hypothesis {\bf (H)}. Let $p\in[p_0,p'_0]$ and $s=d\big|\frac{1}{2}-\frac{1}{p}\big|$. Then the operator $e^{-itH}$ satisfies the estimate
\begin{equation}\label{1.4}
\|e^{-itH}\phi(2^{-k}H)f\|_{L^p}\lesssim(1+2^{k}|t|)^{s}\|f\|_{L^p},\quad k\in\bZ,\ t\in\R,
\end{equation}
uniformly for $\phi$ in bounded subsets of $C^\infty_c(\R)$. 
\end{theorem}
This result is sharp both for the growth in $t$ and for the loss of derivatives, in the sense that when $H=-\Delta$ (and $k$ is replaced by $2k$) it reduces to that mentioned above, which is sharp. Moreover, it is new even for operators satisfying \eqref{cal} with $m=2$. \par
 Besides the case of the free Laplacian, estimate \eqref{1.4} was first proved for $H=-\Delta+V$ in \cite[Theorem 1.4]{JN95} and in \cite[Theorem 5.2]{JN95} under several assumptions on $V$. Actually, we shall see in Theorem \ref{mainteo3} below that we can consider any operator for the form
\[
H=(i\nabla+A)+V,
\]
with a magnetic potential $A\in L^2_{loc}(\rd;\rd)$, and an electric potential $V$ with positive and negative parts $V_+\in\mathcal{K}_{loc}(\rd)$, $V_-\in \mathcal{K}(\rd)$, where $\mathcal{K}(\rd)$ is the Kato class (see Definition \ref{defkato} below). Although this operator is just bounded from below, the conclusions of Theorem \ref{mainteo} still hold for $H$ at least for $k\geq0$.\par
 The proof is inspired by \cite{JN95}.  In short, by duality and interpolation we are reduced to prove the result for $p={p_0}$. Then we exploit the smoothing effect \eqref{uno} of the heat operator  to reduce matters to a continuity result in certain amalgam spaces of functions with an upgraded $L^2$ local regularity ($p_0<2$) and an $\ell^{p_0}$ decay at infinity (on average). Technically, we need a version of these spaces adpapted to different scales, as in \cite{tao}. Then, the off-diagonal decay \eqref{due} of the heat operator is used to prove the desired estimates in such spaces. \par
As an intermediate step, we prove strong $(p_0,p_0)$ estimates for $\phi(H)$, $\phi\in C^\infty_c(\R)$. It would be interesting to know whether the assumption ${\bf (H)}$ (or a variant of it) is still sufficient  for more general spectral multipliers theorems, where $\phi$ satisfies H\"ormander type conditions \cite{CD,DOS02,heb,KU12}. We plan to investigate this issue in a subsequent work. \par
As another remark, we observe that we could allow an exponential factor $\exp(ct)$, $c>0$, in the right-hand sides of \eqref{uno} and \eqref{due}, provided the conclusion of Theorem \ref{mainteo} is restricted to $k\geq0$ (see Section 5.1). \par
As a consequence of Theorem \ref{mainteo} we obtain, by a standard scaling argument \cite[pag. 193]{JN94}, an estimate in Sobolev spaces adapted to the operator $H$. 
\begin{corollary}
Assume the hypothesis {\bf (H)}. Let $p\in[p_0,p'_0]$ and $s=d\big|\frac{1}{2}-\frac{1}{p}\big|$. Then for every $\epsilon>0$ the operator $e^{-itH}$ satisfies the estimates
\begin{equation}\label{1.4bis}
\|e^{-itH}(I+H)^{-s-\epsilon}f\|_{L^p}\lesssim(1+|t|)^{s}\|f\|_{L^p},\quad t\in\R.
\end{equation}
\end{corollary}
This result improves \cite[Theorem 1.4]{JN94} and (at least in the Euclidean setting) \cite[Theorem 5.2]{CCO} and \cite[Theorem 1.3 (b)]{Blu07}, where an additional $\epsilon$-loss occured in the exponent of $t$ (and moreover the stronger estimates \eqref{gge} were assumed). \par
Another interesting issue is the validity of \eqref{1.4bis} with $\epsilon=0$. Indeed, for $H=-\Delta$ in $\R^d$ and $1<p<\infty$, the estimate \eqref{1.4bis} was proved with $\epsilon=0$ (and $t=1$) in \cite{miyachi}, but this sharp form seems out of reach in the present generality, even for fixed $t$. However, using some results of Time-frequency Analysis we will see that estimates such as \eqref{1.4bis} with $\epsilon=0$ indeed hold for a large class of propagators, essentially any operator which is bounded in the so-called modulation spaces (see Section 5 for the definition). This is just a remark, which however seems to be new. Details and examples are given in Section 5.\par
 % As we will see, this is indeed an easy corollary of the continuity of $A$ on modulation spaces \cite{}.\par
The paper is organized as follows. In Section 2 we prove some preliminary results and we define the above mentioned amalgam spaces, together with a criterion of boundedenss. In Section 3 we prove strong $(p,p)$ estimates for the operator $\phi(H)$, for $p\in[p_0,p'_0]$, and $\phi\in C^\infty_c(\R)$. This will be used in the proof of Theorem \ref{mainteo}, which is given in Section 4. Finally in Section 5 we discuss examples of operators which Theorem \ref{mainteo} applies to, and the above mentioned connection with Time-frequency Analysis.
\section{Preliminary results}
\subsection{Some remarks on assumption {\bf (H)}} For future reference, we collect here some comments on the assumption {\bf (H)}.
\begin{remark}\label{rem1}
Let us notice that for a linear operator $A$ and $1\leq p,q \leq \infty$ we have 
\[
\|\mathbf{1}_Q A\|_{L^p\to L^q}\leq \|\mathbf{1}_{Q'} A\|_{L^p\to L^q},\qquad \|A\mathbf{1}_Q\|_{L^p\to L^q}\leq \|A\mathbf{1}_{Q'}\|_{L^p\to L^q}
\]
if $Q\subset Q'$ are measurable sets. Moreover, if $\mathbf{1}_Q=\sum_{k=1}^m \mathbf{1}_{Q_k}$ (pointwise almost everywhere) then 
\[
\|\mathbf{1}_Q A\|_{L^p\to L^q}\leq \sum_{k=1}^m\|\mathbf{1}_{Q_k} A\|_{L^p\to L^q},\quad \|A\mathbf{1}_Q\|_{L^p\to L^q}\leq \sum_{k=1}^m\|A\mathbf{1}_{Q_k} \|_{L^p\to L^q}.
\]
This implies that, for any given $M\in\bN$ the estimates \eqref{uno} and \eqref{due} hold for every $j$ with $2^{-j-M}\leq t^{1/m}<2^{-j+M}$, where the constant implicit in the notation $\lesssim$ will depend on $M$. 
\end{remark}
\begin{remark}\label{rem2}
The estimate \eqref{uno} is equivalent to the couple of estimates
\begin{equation}\label{tre}
\sup_{Q'\in\cQ_j}\sum_{Q\in\cQ_j} \|\mathbf{1}_Q e^{-tH} \mathbf{1}_{Q'}\|_{L^{p_0}\to L^{2}}\lesssim 2^{jd\big(\frac{1}{p_0}-\frac{1}{2}\big)}
\end{equation}
and
\begin{equation}\label{quattro}
\sup_{Q\in\cQ_j}\sum_{Q'\in\cQ_j} \|\mathbf{1}_Q e^{-tH} \mathbf{1}_{Q'}\|_{L^{p_0}\to L^{2}}\lesssim 2^{jd\big(\frac{1}{p_0}-\frac{1}{2}\big)}.
\end{equation}
Indeed, \eqref{tre} follows from \eqref{uno} and H\"older's inequality, because $|Q|=2^{-jd}$. Moreover, if \eqref{uno} holds as stated then it also holds, by duality, with $Q$ and $Q'$ exchanged in the sum and supremum, so that \eqref{tre} holds with the same exchange, which is \eqref{quattro}.\par
Viceversa, assume \eqref{tre} and \eqref{quattro}. Then we have 
\begin{align*}
 \|\mathbf{1}_Q e^{-tH} \mathbf{1}_{Q'}\|_{L^{p_0}\to L^{p'_0}}&\leq\sum_{Q''\in\cQ_j} \|\mathbf{1}_Q e^{-(t/2)H} \mathbf{1}_{Q''}e^{-(t/2)H}\mathbf{1}_{Q'}\|_{L^{p_0}\to L^{p'_0}}\\
 &\leq\sum_{Q''\in\cQ_j} \|\mathbf{1}_Q e^{-(t/2)H} \mathbf{1}_{Q''}\|_{L^{2}\to L^{p'_0}} \|\mathbf{1}_{Q''}e^{-(t/2)H}\mathbf{1}_{Q'}\|_{L^{p_0}\to L^{2}},
\end{align*}
and \eqref{uno} follows from Remark \ref{rem1}, \eqref{tre} and the dual version of \eqref{quattro}. 
\end{remark}
\begin{remark}\label{rem3} The estimate \eqref{due} implies, by duality, 
\begin{equation}\label{duebis}
\sup_{Q\in\cQ_j}\sum_{Q'\in\cQ_j} (1+2^j{\rm dist}\,(Q,Q'))^{N}\|\mathbf{1}_Q e^{-tH} \mathbf{1}_{Q'}\|_{L^{2}\to L^2}\lesssim 1.
\end{equation}
\end{remark}
\begin{remark}\label{rem4} Let
\[
S_\lambda f(x)=f(\lambda x),\quad\lambda>0.
\]
Observe that if the assumption ${\bf (H)}$ holds for the operator $H$ then it also holds for the operator
\[
H_k:= 2^k S_{\lambda_k}H S_{\lambda_k}^{-1},\qquad \lambda_k=2^{k/m},\quad k\in\bZ,
\]
uniformly with respect to $k$. \par
Indeed, by the spectral calculus we have
\[
e^{-tH_k}=S_{\lambda_k} e^{-2^k t H} S_{\lambda_k}^{-1}
\]
and for $Q,Q'\in\cQ_j$, with $2^{-j}\leq t^{1/m}<2^{-j+1}$, we have 
\[
\mathbf{1}_Q e^{-tH_k} \mathbf{1}_{Q'}=S_{\lambda_k} \mathbf{1}_{\lambda_k Q} e^{-2^k t H} \mathbf{1}_{\lambda^k Q'}S_{\lambda_k}^{-1},
\]
with $\lambda Q:=\{\lambda x:\, x\in Q\}$.\par Moreover, for a linear operator $A$ and $1\leq p,q\leq \infty$ we have 
\[
\|S_\lambda A S_\lambda^{-1}\|_{L^p\to L^q}=\lambda^{d\big(\frac{1}{p}-\frac{1}{q}\big)},\quad \lambda>0.
\]
Hence, it is sufficent to apply the assumption ${\bf (H)}$ with $t$ replaced by $2^k t$. More precisely, the involved cubes should be those dyadic having sidelength $2^{-M}$, with $2^{-M}\leq (2^k t)^{1/m}<2^{-M+1}$, but one can cover each of the cubes $\lambda_k Q'$ and $\lambda_k Q'$ (whose sidelength is exactly $(2^k t)^{1/m}$) by using $2^d$ of such dyadic cubes (cf. Remark \ref{rem1}).
\end{remark}

\subsection{Amalgam spaces \cite{dav95,heil,JN94,tao}} For $1\leq p,q\leq\infty$, $j\in\mathbb{Z}$, consider the space $X^{p,q}_j$ of measurable functions in $\rd$ equipped with the norm
\[
\|f\|_{X^{p,q}_j}:=\Big(\sum_{Q\in\cQ_j} \|\mathbf{1}_{Q}f\|_{L^q}^p\Big)^{1/p}
\]
(with obvious changes if $q=\infty$). As above $\cQ_j$ denotes the collection of dyadic cubes of sidelength $2^{-j}$.  We also set $X^{p,q}=X^{p,q}_0$.\par
Notice that $X^{p,p}_j=L^p$ for every $j$ and $p$. Moreover we will need the following embeddings.
\begin{proposition}\label{pro5}
For $1\leq p\leq q\leq\infty$, $j\in\bZ$, we have 
\begin{equation}\label{pro5eq1}
\|f\|_{X^{p,q}}\leq \max\{1,2^{-jd\big(\frac{1}{p}-\frac{1}{q} \big)}\}\|f\|_{X^{p,q}_j}.
\end{equation}
\end{proposition}
\begin{proof}
Consider first the case $j<0$. Then we prove that, if $\tilde{Q}\in\cQ_{j}$, 
\[
\Big(\sum_{\cQ_0\ni Q\subset \tilde{Q}}\|\mathbf{1}_Q f\|^p_{L^q}\Big)^{1/p}\leq 2^{-jd\big(\frac{1}{p}-\frac{1}{q} \big)}\|\mathbf{1}_{\tilde{Q}} f\|_{L^q}.
\]
Since \[
\|\mathbf{1}_{\tilde{Q}} f\|_{L^q}=\Big(\sum_{\cQ_0\ni Q\subset \tilde{Q}}\|\mathbf{1}_Q f\|^q_{L^q}\Big)^{1/q},
\]
 the result follows from H\"older's inequality for finite sequences, with $N=2^{-jd}$ elements.\par
Consider now the case $j\geq0$. We prove that, if $\tilde{Q}\in\cQ_0$, 
\[
\|\mathbf{1}_{\tilde{Q}} f\|_{L^q}\leq \Big(\sum_{\cQ_j\ni Q\subset \tilde{Q}}\|\mathbf{1}_Q f\|^p_{L^q}\Big)^{1/p}.
\]
Again, the left-hand side is equal to $\Big(\sum_{\cQ_j\ni Q\subset \tilde{Q}}\|\mathbf{1}_Q f\|^q_{L^q}\Big)^{1/q}$, and the result follows, because $p\leq q$.
\end{proof}

Here is an elementary criterion for boundedness on $X^{p,q}$.
\begin{proposition}\label{pro6}
Let $A$ be a linear operators satisfying, for some $1\leq q_1,q_2\leq\infty$, $j\in\mathbb{Z}$,
\[
\sup_{Q'\in\cQ_j}\sum_{Q\in\cQ_j}\|\mathbf{1}_Q A \mathbf{1}_{Q'}\|_{L^{q_1}\to L^{q_2}}=M_1<\infty,
\]
and 
\[
\sup_{Q\in\cQ_j}\sum_{Q'\in\cQ_j}\|\mathbf{1}_Q A \mathbf{1}_{Q'}\|_{L^{q_1}\to L^{q_2}}=M_2<\infty.
\]

Then, for every $1\leq p\leq\infty$,
\[
\|Af\|_{X_j^{p,q_2}}\leq M_1^{1-\theta}M_2^\theta \|f\|_{X_j^{p,q_1}},\quad \frac{1}{p}=1-\theta.
\]
\begin{proof}
The desired estimate follows at once from the definition of the spaces $X^{p,q}_j$ and Schur's test for operators acting on sequences. 
\end{proof}
\end{proposition} 

\subsection{A criterion for boundedness on $X^{p,2}$ \cite{JN94}}

We recall here a result from \cite{JN94} which gives a sufficient condition for a linear operator $A$ to be bounded on the spaces $X^{p,2}$, $1\leq p\leq 2$. 
\begin{theorem}\label{teo7}
Let $A$ be a bounded operator on $L^2$, and for any $l=1,\ldots, d$, define the commutator ${\rm Ad}_l(A)=[x_l,A]$. Suppose that for some $M\geq1$ we have  
\[
\|{\rm Ad}_l^k(A)\|_{L^2\to L^2}\leq M^k,\quad 0\leq k\leq \lfloor d/2\rfloor+1, \quad 1\leq l\leq d.
\]
Then, for $1\leq p\leq 2$,
\begin{equation}\label{limitl1}
\|Af\|_{X^{p,2}}\leq C_0M^{d\big(\frac{1}{p}-\frac{1}{2} \big)} \|f\|_{X^{p,2}}
\end{equation}
where the constant $C_0$ depends only on $d$ and upper bounds for $\|A\|_{L^2\to L^2}$.
\end{theorem}

\begin{proof}
The proof can be found in \cite{JN94}, but the result is not explicitly stated there in the present form. Hence, for the benefit of the reader we point out detailed references of the main steps.\par
By the computations at the end of \cite[page 261]{JN95}, the assumption implies the $L^2$ weighted estimate
\[
\|\langle \cdot-n\rangle^k A \langle \cdot-n\rangle^{-k}\|_{L^2\to L^2}\lesssim_d M^k, \quad \forall n\in\mathbb{Z}^d,\quad k=\lfloor d/2\rfloor+1,
\]
 By \cite[Formula (3.7)]{JN95} this last formula implies that
\[
|||A|||_{k}:=\|A\|_{L^2\to L^2}+\sup_{n\in {\bZ}^d} \|\langle \cdot-n\rangle^k A \mathbf{1}_{Q_n}\|_{L^2\to L^2} \lesssim_d M^{k},\quad k=\lfloor d/2\rfloor+1,
\]
where $Q_n=n+[0,1]^d$. Since $k>d/2$ we can apply the interpolation inequality in \cite[Theorem 2.4]{JN95} with $\beta=k$ and we obtain 
\[
\|Af\|_{X^{1,2}}\lesssim_d \|A\|_{L^2\to L^2}^{1-{d/(2(\lfloor d/2\rfloor+1))}}M^{d/2} \|f\|_{X^{1,2}}.
\]
By interpolation with the $L^2\to L^2$ estimate, we deduce
\eqref{limitl1}.

\end{proof}
\section{Boundedness of $\phi(H)$}
This section is devoted to the proof of the following result, which is an intermediate step for Theorem \ref{mainteo}.
\begin{theorem}\label{mainteo0}
Let $H$ satisfy the assumption ${\bf (H)}$. Then for every $p\in[p_0,p'_0]$ we have 
\begin{equation}\label{eq1}
\|\phi(2^k H) f\|_{L^p}\lesssim\|f\|_{L^p}.
\end{equation}
uniformly for $\phi$ in bounded subsets of $C^\infty_c(\R)$ and $k\in\bZ$.
\end{theorem}
First of all we observe that it suffices to consider the case $k=0$. Indeed, as observed in Remark \ref{rem4}, the same assumption {\bf (H)} holds for the operator 
\[
H_k:= 2^k S_{\lambda_k}H S_{\lambda_k}^{-1},\qquad \lambda_k=2^{k/m},\quad k\in\bZ,
\]
uniformly with respect to $k$.
On the other hand, by the spectral calculus 
\[
f(H_k)=S_{\lambda_k}f(2^k H) S_{\lambda_k}^{-1}
\]
and intertwining with $S_{\lambda_k}$ preserves the $(p,p)$ norm. \par
Let us therefore prove \eqref{eq1} for $k=0$. \par
We begin with an easy lemma.
\begin{proposition}\label{pro7}
We have the estimates
\[
\|e^{-tH}f\|_{X^{p_0,2}}\lesssim (1+t^{-\frac{d}{m}\big(\frac{1}{p_0}-\frac{1}{2} \big)})\|f\|_{L^{p_0}},\quad t>0.
\]
\end{proposition}
\begin{proof}
By \eqref{uno}, Remark \ref{rem2} and Proposition \ref{pro6} we have, for $2^{-j}\leq t^{1/m}<2^{-j+1}$, 
\[
\|e^{-tH}f\|_{X^{p_0,2}_j}\lesssim 2^{dj\big(\frac{1}{p_0}-\frac{1}{2} \big)}\|f\|_{X^{p_0,p_0}_j}.
\]
Since $X^{p_0,p_0}_j=L^{p_0}$, the embedding in Proposition \ref{pro5} gives the desired conclusion.

\end{proof}

Consider now the resolvent operator
\[
R=(I+H)^{-1}.
\]
\begin{proposition}\label{pro8}
If $\beta>\frac{d}{m}\big(\frac{1}{p_0}-\frac{1}{2}\big)$ we have
\[
\|R^\beta f\|_{X^{p_0,2}}\lesssim \|f\|_{L^{p_0}}.
\]
\end{proposition}
\begin{proof}
It is sufficient to use the formula 
\begin{equation}\label{eqint}
R^\beta=\frac{1}{\Gamma(\beta)}\int_0^{+\infty} t^{\beta-1} e^{-t} e^{-tH}\, dt,
\end{equation}
together with Proposition \ref{pro7} and  Minkowski's inequality for integrals.
\end{proof}

The criterion in Theorem \ref{teo7} allows one to transfer estimates in amalgam spaces from $R$ to $e^{-i\xi R}$, $\xi\in\R$.
\begin{proposition}\label{pro9}
We have 
\[
\|e^{-i\xi R}f\|_{X^{p_0,2}}\lesssim (1+|\xi|)^{d(1/p_0-1/2)}\|f\|_{X^{p_0,2}},\quad \xi\in\R.
\]
\end{proposition}
\begin{proof}
By the criterion in Theorem \ref{teo7}, it is sufficient to prove that 
\[
\|{\rm Ad}^k_l(e^{-i\xi R})f\|_{L^2}\lesssim (1+|\xi|)^{k}\|f\|_{L^2}
\]
for $k=0,1,\ldots,\lfloor d/2\rfloor+1$, $l=1,\ldots,d$. \par
 Using the formula 
\[
{\rm Ad}_l^1 (e^{-i\xi R})=-i\int_0^\xi e^{-isR}{\rm Ad}_l^1(R) e^{-i(\xi-s)R}\, ds
\]
repeatedly, we are reduced to prove that
\begin{equation}\label{eq15}
\|{\rm Ad}_l^k(R)u\|_{L^2}\lesssim \|u\|_{L^2}
\end{equation}
for $k=0,1,\ldots,\lfloor d/2\rfloor+1$.\par
Now, setting $R(x,y)$ for the distribution kernel of $R$, that of the operator ${\rm Ad}_l^k(R)$ is $ (x_l-y_l)^k R(x,y)$. Using the integral representation \eqref{eqint} with $\beta=1$, and Minkowski's inequality for  integrals we see that it is sufficient to prove that, for $2^{-j}\leq t^{1/m}<2^{-j+1}$, the operator
\[
Af(x):=2^{jk} \int_{\rd} (x_l-y_l)^k p_t(x,y) f(y)\, dy
\]
is bounded on $L^2=X^{2,2}_j$ uniformly with respect to $j\in\bZ$, for $k=1,2,\ldots,\lfloor d/2\rfloor+1$ where $p_t(x,y)$ denotes the distribution kernel of $e^{-tH}$ (in fact, the exponential factor in \eqref{eqint} compensates the factor $2^{-jk}\leq t^{k/m}$ which has been introduced).  \par
Now, to prove the boundedness of $A$ on $L^2=X^{2,2}_j$, we use Proposition \ref{pro6}.  Hence, by self-adjointness we are reduced to prove that
\begin{equation}\label{eq4}
\sup_{Q'\in\cQ_j}\sum_{Q\in\cQ_j}\|\mathbf{1}_Q A \mathbf{1}_{Q'}\|_{L^{2}\to L^{2}}\leq C
\end{equation}
for a constant $C$ independent of $j$.\par
To this end, observe that, if $Q=z_Q+2^{-j}[0,1]^d$, $Q'=z_{Q'}+2^{-j}[0,1]^d$, then
\begin{multline}
\mathbf{1}_Q A \mathbf{1}_{Q'}
=\sum_{\alpha+\beta+\gamma=k}\frac{k!}{\alpha!\beta!\gamma!} (2^j(z_{Q,l}-z_{Q',l}))^\alpha\\
\times (2^j(x_{l}-z_{Q,l}))^\beta \mathbf{1}_Q e^{-tH}\mathbf{1}_{Q'} (2^j(z_{Q',l}-y_{l}))^\gamma.
\end{multline}
Formula \eqref{eq4} then follows from the assumption \eqref{due} and the elementary estimates
\[
\|(2^j(x_{l}-z_{Q,l}))^\beta\mathbf{1}_Q\|_{L^2\to L^2}\lesssim 1,\qquad \|\mathbf{1}_{Q'} (2^j(z_{Q',l}-y_{l}))^\gamma\|_{L^2\to L^2}\lesssim 1.
\]
\end{proof}

Now we continue with the proof of Theorem \ref{mainteo0}.
Writing 
\[
\phi(R)=(2\pi)^{-1} \int_{-\infty}^{+\infty} e^{i\xi R} \widehat{\phi}(\xi)\, d\xi
\]
and using Proposition \ref{pro9} we deduce that
\begin{equation}\label{eq16}
\|\phi(R)f\|_{X^{p_0,2}}\lesssim \|(1+|\xi|)^{d(1/p_0-1/2)}\widehat{\phi}\|_{L^1} \|f\|_{X^{p_0,2}}.
\end{equation}
Now, given $\phi\in C^\infty_0(\R)$ we can find $\psi\in C^\infty_0(\R_+)$ such that $\psi((\lambda+1)^{-1})=\phi(\lambda)$ for $\lambda\geq0$, so that $\phi(H)=\psi(R)$ is bounded on $X^{p_0,2}$. Finally, for $\phi\in C^\infty_0(\R)$ we write 
\[
\phi(H)=R^{-\beta} \phi(H)\cdot R^\beta,
\]
 with $\beta>\frac{d}{m}\big(\frac{1}{p_0}-\frac{1}{2}\big)$. The operator $R^\beta$ is then bounded $L^{p_0}\to X^{p_0,2}$ by Proposition \ref{pro8}, whereas the first factor $R^{-\beta} \phi(H)$ is bounded on $X^{p_0,2}$ by what we just proved (applied to a function in $C^\infty_c(\R)$ with is equal to $(\lambda+1)^\beta \phi(\lambda)$ for $\lambda\geq0$). Since $X^{p_0,2}\hookrightarrow L^{p_0}$, $f(H)$ is bounded on $L^{p_0}$, and therefore on every $L^p$, $p\in[p_0,p'_0]$, by duality and interpolation with the $L^2$ case.\par
 The uniformity of the estimate when $\phi$ varies in bounded subsets of $C^\infty_c(\R)$ also follows from \eqref{eq16} and this last argument.

This concludes the proof of Theorem \ref{mainteo0}.

\section{Proof of the main result (Theorem \ref{mainteo})}

By using the same scaling argument as in the proof of Theorem \ref{mainteo0} 
 it is sufficient to prove that
\begin{equation}\label{1.5}
\|e^{-it H} \phi(H)\|_{L^p\to L^p}\leq C(1+|t|)^{d |\frac{1}{2}-\frac{1}{p}|},
\end{equation}
uniformly for $\phi$ in bounded subsets of $C^\infty_0(\R)$. \par
We will adopt the notation ${\rm Ad}(A)=[x_j,A]$, $j=1,\ldots,d$, for the commutator $[x_j,A]$, as in Theorem \ref{teo7}, omitting the subscript $j$ for simplicity. Below we will prove that, for $k=\lfloor d/2\rfloor+1$,
\begin{equation}\label{1.6}
\|{\rm Ad}^l (R^{2^{k+1}-2}e^{-itH})\|_{L^2\to L^2}\leq C(1+|t|)^{l},\qquad l\leq k.
\end{equation}
This implies, from Theorem \ref{teo7}, that
\[
\| R^{2^{k+1}-2}e^{-itH}\|_{X^{p_0,2}\to X^{p_0,2}}\leq C (1+|t|)^{d(1/p_0-1/2)}.
\]
Combining this estimate with Proposition \ref{pro8} and Theorem \ref{mainteo0}  we obtain 
\begin{align*}
\|e^{-itH} &\phi(H)\|_{L^{p_0}\to X^{p_0,2}}\\
&\lesssim\| R^{2^{k+1}-2}e^{-itH}\|_{X^{p_0,2}\to X^{p_0,2}}  \|R^{\beta-2^{k+1}+2} \|_{L^{p_0}\to X^{p_0,2}} \|R^{-\beta} \phi(H)\|_{L^{p_0}\to L^{p_0}}\\
&\lesssim (1+|t|)^{d(1/p_0-1/2)},
\end{align*}
where we choose 
\[
\beta>2^{k+1}-2+\frac{d}{m}\big(\frac{1}{p_0}-\frac{1}{2}\big)
\]
in order to apply Proposition \ref{pro8}. Using the inclusion $X^{p_0,2}\hookrightarrow L^{p_0}$ we deduce \eqref{1.5} for $p=p_0$.
 By duality and interpolation with the $L^2$ case we get \eqref{1.5}.\par
It remains to prove \eqref{1.6}. This will be done on induction on $k=0,1,\ldots,\lfloor d/2\rfloor+1$.\par
First we prove the following result.
\begin{proposition}\label{pro1.3}
For every $k=1,\ldots,\lfloor d/2\rfloor+1$, the operator ${\rm Ad}(R^{2^{k+1}-2}e^{-itH})$ is given by a (finite) linear combination of operators of the following type:
\[
 R^{\mu_1}\, R^{2^k-2} e^{-itH} {\rm Ad}(R^{\mu_2})\,R^{\mu_3},\qquad \mu_1,\mu_2,\mu_3\in\mathbb{N},
\]
\[
R^{\mu_1}\, {\rm Ad}(R^{\mu_2})\,R^{2^k-2} e^{-itH}\,  R^{\mu_3},\qquad \mu_1,\mu_2,\mu_3\in\mathbb{N},
\]
and
\[
\int_0^t R^{2^k-2}\, e^{-isH}\, {\rm Ad}(R)\, R^{2^k-2}
e^{-i(t-s)H}\,ds.
\]
\end{proposition}
\begin{proof}
The result is obtained by induction on $k$. It is true for $k=1$. Indeed,
\[
{\rm Ad}(Re^{-itH}R)= {\rm Ad}(R) e^{-itH}R +Re^{-itH}{\rm
Ad}(R)-i\int_0^t e^{-isH}R[x_j,H]Re^{-i(t-s)H}\,ds.
\]
Now, $R[x_j,H]R=[R,x_j]=-{\rm Ad}(R)$, so that the result for $k=1$ is verified. \par
 Let us assume it holds for $k-1$ and compute
\begin{align*}
{\rm Ad} (R^{2^{k+1}-2}e^{-itH})&= {\rm Ad} (R^{2^{k-1}}(R^{2^{k}-2}  e^{-itH})R^{2^{k-1}})\\
&={\rm Ad}(R^{2^{k-1}})\  R^{2^{k}-2}  e^{-itH}\ R^{2^{k-1}}+R^{2^{k-1}}\ {\rm Ad}(R^{2^{k}-2}  e^{-itH})  \ R^{2^{k-1}}\\
&\qquad\qquad\qquad\qquad+  R^{2^{k-1}}  \ R^{2^{k}-2}  e^{-itH}\ {\rm Ad} (R^{2^{k-1}}).
\end{align*}
The first and the last term are of the desired form. The second one  is also of the desired form by the inductive hypothesis, because
\[
R^{2^{k-1}}R^{2^{k-1}-2}=R^{2^k-2}.
\]
\end{proof}

Let us now prove \eqref{1.6} by induction on $k=0,1,\ldots,\lfloor d/2\rfloor+1$. The result is trivially true for $k=0$. Assume it holds for $k-1$, and write, for $l\leq k$,
\[
{\rm Ad}^l(R^{2^{k+1}-1}e^{-itH})= {\rm Ad}^{l-1}\left({\rm Ad}(R^{2^{k+1}-1}e^{-itH})\right).
\]
Using the above Proposition \ref{pro1.3}, the formula
\[
{\rm Ad}^{l-1}(A_1\cdots A_n)=\sum_{m_1+\ldots+m_n=l-1}\frac{(l-1)!}{m_1!\cdots m_n!}{\rm Ad}^{m_1}(A_1)\cdots {\rm Ad}^{m_n}(A_n),
\]
as well as the inductive hypothesis and \eqref{eq15} we obtain \eqref{1.6}.

\section{Examples and concluding remarks}
\subsection{Schr\"odinger operators} Here is our main example. 
We recall the definition of the Kato class from \cite[pag. 453]{simon}.
\begin{definition}\label{defkato}
 A real-valued measurable function in $\rd$ is called to lie in the Kato class $\mathcal{K}(\rd)$ if and only if
\begin{itemize}
\item[a)] If $d=3$,
\[
\lim_{r\downarrow 0} \sup_{x\in\rd}\int_{|x-y|<r}\frac{|V(y)|}{|x-y|^{d-2}}\, dy=0.
\]
\item[b)]  If $d=2$
\[
\lim_{r\downarrow 0} \sup_{x\in\rd}\int_{|x-y|<r}\log|x-y|^{-1}|V(y)|\, dy=0.
\]
\item[c)] If $d=1$
\[
\sup_{x\in\rd}\int_{|x-y|<1}|V(y)|\, dy<\infty.
\]
We moreover define $\mathcal{K}_{\rm loc}(\rd)$ as the space of functions $V$ such that $V{\bf 1}_B\in \mathcal{K}(\rd)$ for every ball $B$.
\end{itemize}
\end{definition}
It follows from H\"older inequality that $L^p_{unif}(\rd)\subset \mathcal{K}(\rd)$ if $p>d/2$ ($d\geq2$), where the uniform $L^p$ spaces are defined by the norm
\[
\|V\|^p_{L^p_{unif}}=\sup_{x\in\rd}\int_{|x-y|<1}|V(y)|^p\, dy<\infty.
\]
For example, $1/|x|^\alpha$ belongs to $\mathcal{K}(\rd)$ if $0<\alpha<2$. \par
Now, we have the following result.
\begin{theorem}\label{mainteo3}
Consider the operator 
\[
H=(i \nabla-A)^2+V,
\]
 where $A\in L^2_{loc}(\rd;\rd)$, $V=V_+ -V_-$ (positive and negative parts), 
 \[
 V_+\in\mathcal{K}_{loc}(\rd),\quad V_-\in \mathcal{K}(\rd).
 \] Then $H$ has a self-adjoint extension in $L^2(\rd)$, bounded from below. \par
 Moreover for $1\leq p\leq\infty$, $s=d\big|\frac{1}{p}-\frac{1}{2}\big|$, the operator $e^{-itH}$ satisfies the estimate
\begin{equation}
\|e^{-itH}\phi(2^{-k}H)f\|_{L^p}\lesssim(1+2^{k}|t|)^{s}\|f\|_{L^p},\quad k\geq 0,\ t\in\R,
\end{equation}
uniformly for $\phi$ in bounded subsets of $C^\infty_c(\R)$. 
\end{theorem} 
\begin{proof}
The first part of the statement is proved in \cite[Sections B1, B13]{simon}. Moreover, combining the estimates for the heat kernel in \cite[Proposition B.6.7]{simon} (case $A\equiv0$) with the diamagnetic inequality (\cite[Theorem B.13.2]{simon}), one sees that the operator $e^{-tH}$ has a measurable kernel $p_t(x,y)$ satisfying 
\[
|p_t(x,y)|\lesssim t^{-d/2} \exp(ct)\exp\Big(-\frac{|x-y|^{2}}{4t}\Big),\quad \ t>0,\ x,y\in\rd,
\]
for some $c>0$. Hence, for some $c'>0$, the operator $H+c' I$ in non-negative and satisfies the assumption {\bf (H)} with $p_0=1$, $m=2$. We can then apply Theorem \ref{mainteo} to the operator $H+c'I$, with the cut-off $\phi(x-2^{-k}c')$ and obtain the desired conclusion for $k\geq0$, since the functions $\phi(x-2^{-k}c')$ for $k\geq0$ vary in a bounded subset of $C^\infty_c(\R)$ if $\phi$ does. 
\end{proof}
\subsection{Operators with variable coefficients}
Following \cite{dancona}, consider the operator 
\[
-Hf=\nabla^b\cdot(a(x)\nabla^b f)-c(x) f,\qquad \nabla^b=\nabla+ib(x)
\]
where $a(x)=[a_{jk}(x)]_{j,k=1}^d$, $b(x)=(b_1(x),\ldots,b_d(x))$ and $c(x)$ satisfy:
\[
\textrm{$a,b,c$ are real-valued, $a_{jk}=a_{kj}$ and $NI\geq a(x)\geq \nu I$ for some $N\geq\nu>0$}.
\]
Suppose moreover $d\geq 3$, and the following conditions in terms of Lorentz spaces:
\[
a\in L^\infty,\ b\in L^4_{loc}\cap L^{d,\infty},\ \nabla \cdot b\in L^2_{loc},\ c\in L^{d/2,1},\ \|c_{-}\|_{L^{d/2,1}}<\epsilon.
\]
 Then if $\epsilon>0$ is sufficiently small the operator $H$ extends to a non-negative selfadjoint operator in $L^2(\rd)$ and its heat kernel $e^{-tH}$ satisfies the Gaussian bound \[
|p_t(x,y)|\lesssim t^{-d/2}\exp\Big(-\frac{|x-y|^{2}}{Ct}\Big),\quad \ t>0,\ x,y\in\rd,
\]
for some $C>0$; see \cite[Propositions 6.1, 6.2]{dancona}. Hence Theorem \ref{mainteo} applies with $p_0=1$.\par
 We also refer to \cite{DP05,DFV10,O04} for early results on Gaussian bounds for variable coefficient second-order operators.
\subsection{Higher order operators (constant coefficients)} Consider the operator 
\[
H=(-\Delta)^k,\quad k\in\bN,\ k\geq 1.
\]
 Let us verify that the assumption {\bf (H)} is verified with $p_0=1$, $m=2k$. It is sufficient to show that the heat kernel $p_t(x,y)$ satisfies the estimate
\[
|p_t(x,y)|\lesssim t^{-d/(2k)}\exp\big(-b(t^{-1/{(2k)}}|x-y|)^{\frac{2k}{2k-1}}\big),\quad \ t>0,\ x,y\in\rd,
\]
for some $b>0$. This is known, but we provide here a proof for the sake of completeness and also because this method of proof extends to any elliptic, homogeneous and non-negative constant coefficient operator. We leave this generalization to the interested reader.\par By homogeneity we are reduced to prove that, if $f(\xi)=\exp\big({-|\xi|^{2k}}\big)$, then its inverse Fourier transform verifies
\[
|{f}^{\vee} (x)|\lesssim \exp \big(-b|x|^{\frac{2k}{2k-1}}\big),\quad \ x\in\rd.
\]
It is easy to see that this inequality is equivalent to
\[
|x^\alpha {f}^{\vee}(x)|\leq C^{|\alpha|+1}(\alpha!)^{1-\frac{1}{2k}}, \quad\alpha\in\bN^d,\ x\in\rd,
\]
for some constant $C>0$ independent of $\alpha$ (see e.g. \cite[Proposition 6.1.7]{nicola}). This in turn follows if we prove that
\begin{equation}\label{des}
\|\partial^\alpha f\|_{L^1}\leq C^{|\alpha|+1}(\alpha!)^{1-\frac{1}{2k}}, \quad\alpha\in\bN^d.
\end{equation}
 This can be verified by means of the Cauchy estimates, as follows.\par Observe that $f$ has an entire extension $f(\zeta)=\exp(-(\zeta_1^2+\ldots+\zeta_d^2)^k)$, $\zeta=(\zeta_1,\ldots,\zeta_d)\in\bC^d$, satisfying
\begin{equation}\label{hpm2}
| f(\zeta)|\leq \exp\Big({-(1/2)|{\rm Re\, \zeta}|^{2k}+C|{\rm Im}\, \zeta|^{2k}}\Big),\quad \zeta\in\mathbb{C}^d.
\end{equation}
Therefore, given $\xi=(\xi_1,\ldots,\xi_d)\in\rd$, we consider  the polydisk 
\[
\tilde{B}(\xi,R)=\prod_{j=1}^d B(\xi_j,R)=\{\zeta\in\mathbb{C}^d:\, |\zeta_j-\xi_j|\leq R\},
\]
 with $R=(1+|\alpha|)^{1/(2k)}$, $\alpha\in\bN^d$. Observe that \eqref{hpm2} implies
\begin{equation}\label{E1}
\sup_{\zeta\in \tilde{B}(\xi,R)}|f(\zeta)|\leq e^{-(1/2)(|\xi|-R)_+^{2k}+C(\sqrt{d}R)^{2k}},
\end{equation}
where $(\cdot)_+$ denotes the positive part. \par
The Cauchy integral formula
$$\partial^\alpha_{\xi} f(\xi)=\frac{\a!}{(2\pi i)^d}\int\cdots\int_{\partial B(\xi_1,R)\times\cdots\times\partial B(\xi_d,R)}\frac{f(\zeta_1,\dots,\zeta_d)}{(\zeta_1-\xi_1)^{\a_1+1}\cdots(\zeta_d-\xi_d)^{\a_d+1}}d\zeta_1\cdots d\zeta_d
$$
and the estimate \eqref{E1} yield
\begin{equation}\label{E2}
|\partial^\alpha_{\xi} f(\xi)|\leq e^{-(1/2)(|\xi|-R)_+^{2k}}\frac{\alpha! e^{C(\sqrt{d}R)^{2k}}}{R^{|\alpha|}}\leq e^{-(1/2)(|\xi|-R)_+^{2k}} \frac{C_1^{|\alpha|+1}\alpha!}{{(1+|\alpha|)}^{|\alpha|/(2k)}}.
\end{equation}
Using Stirling formula, we have
$$\frac{1}{(1+|\alpha|)^{|\alpha|/(2k)}}\leq \frac{C_2^{|\a|}}{(\a!)^{1/(2k)}},
$$
which combined with \eqref{E2} gives
\[
|\partial^\alpha_{\xi} f(\xi)|\leq C^{|\alpha|+1}(\alpha!)^{1-\frac{1}{2k}}e^{-(1/2)(|\xi|-R)_+^{2k}}, \quad\alpha\in\bN^d.
\]
Integrating separetely for $|\xi|>R$ and $|\xi|\leq R$ gives the desired estimate \eqref{des} (the factor coming from the volume of the ball of radius $R$ is absorbed by taking a slightly bigger constant $C$ in \eqref{des}).
\subsection{Higher order operators (measurable coefficients)} Following \cite{dav95}, let $k\geq1$ and consider the operator 
\[
H=\sum_{|\alpha|\leq k,\,|\beta|\leq k}D^\alpha (a_{\alpha,\beta}(x)D^\beta f),
\]
where $a_{\alpha,\beta}(x)=\overline{a_{\beta,\alpha}(x)}$ are complex-valued, bounded measurable functions. The associated quadratic form is
\[
Q(f,f)=\int_{\rd}\sum_{|\alpha|\leq k,\,|\beta|\leq k}a_{\alpha,\beta}(x)D^\beta f(x)\overline{D^\alpha f(x)}\, dx.
\]
We suppose that $Q=Q_0+Q_1+Q_2$ where $Q_0, Q_1,Q_2$ have the same form as $Q$, but $Q_0$ is homogeneous and elliptic of degree $2k$ and has constant coefficients, $Q_1$ is homogeneous of degree $2k$ and is non-negative in the sense that
\[
\sum_{|\alpha|=k,\,|\beta|=k}a_{1,\alpha,\beta}(x) v_\beta\overline{v_\alpha}\geq0
\]
for all $v_\alpha\in\mathbb{C}$, $x\in\rd$, whereas $Q_2$ contains lower order terms. \par
We distinguish two cases. If $2k>d$ then it was proved in \cite[Lemma 19]{dav95} that the heat kernel satisfies the estimate
\[
|p_t(x,y)|\lesssim t^{-d/(2k)}\exp(ct)\exp\big(-b(t^{-1/{(2k)}}|x-y|)^{\frac{2k}{2k-1}}\big),\quad \ t>0,\ x,y\in\rd,
\]
for some $b,c>0$.\par
In the case $2k<d$ the above estimate can fail. However, let
\[
p_0=\frac{2d}{d+2k}.
\]
Then $p_0<2$, and
\[
\|\mathbf{1}_{B(x,t^{1/(2k)})} e^{-tH} \mathbf{1}_{B(y,t^{1/(2k)})}\|_{L^{2}\to L^{p'_0}}\lesssim t^{-\frac{d}{2k}\big(\frac{1}{2}-\frac{1}{p'_0}\big)}\exp(ct)\exp\Big(-b\big(t^{-1/(2k)}|x-y|\big)^{\frac{2k}{2k-1}}\Big),
\]
for some $b,c>0$; see \cite[Lemma 24]{dav95}. Hence, by arguing as in Remark \ref{rem2} we see that \eqref{gge} holds with $m=2k$ except for a further factor $\exp(ct)$ in the right-hand side.\par
By the same arguments as in Section 5.1 we deduce that in both cases the conclusion of Theorem \ref{mainteo0} holds (if $2k>d$ with $p_0=1$, if $2k<d$ with $p_0=2d/(d+2k)$) at least for all $k\geq0$.

\subsection{Fractional Laplacian}
Consider the fractional Laplacian 
\[
H=(-\Delta)^\alpha,\quad\alpha>0.
\]
By using standard results on homogeneous distributions it is easy to see that its heat kernel satisfies 
\[
0<p_t(x,y)\lesssim t^{-\frac{d}{2\alpha}}(1+t^{-\frac{1}{2\alpha}}|x-y|)^{-(d+2\alpha)},
\]
Hence we see that \eqref{uno} is satisfied with $p_0=1$ for every $\alpha>0$, whereas \eqref{due} holds for $2\alpha> \lfloor d/2\rfloor+1$ (both with $m=2\alpha$). \par
\subsection{Estimates in Sobolev spaces via Time-frequency Analysis}
For $1\leq p\leq\infty$, $s\in\R$, let $L^p_s$ stand for the usual Sobolev (or Bessel potential) space, i.e.
\[
\|f\|_{L^p_s}:=\|(1-\Delta)^s f\|_{L^p}.
\]

In this section we prove some optimal estimates in such spaces, relying simply on known results from Time-frequency Analysis. \par
To be precise, we recall the definition of modulation spaces \cite{fei,book,ruz,baoxiang}. They can be defined similarly to the Besov spaces, but for a different geometry: the dyadic annuli in the frequency domain are replaced by isometric boxes ${Q}_k$, $k\in\bZ^d$, which allows a finer analysis in many respects (see \cite{baoxiang}). Namely, for $1\leq p,q\leq\infty$, $s\geq 0$, one defines
\begin{equation}\label{prima}
M^{p,q}_s=\Big\{f\in\cS'(\rd): \|f\|_{M^{p,q}_s}:=\Big(\sum_{k\in\bZ^d}\langle k\rangle^{sq}\|\square_k f\|_{L^p}^q\Big)^{1/q}<\infty\Big\}
\end{equation}
(with obvious changes if $q=\infty$), where $\square_k$ are Fourier multipliers with symbols $\mathbf{1}_{{Q}_k}$ conveniently smoothed; we also set $M^{p,q}$ for $M^{p,q}_0$. We recapture in particular the $L^2$-based Sobolev spaces $M^{2,2}_s=H^s$, whereas the space $M^{\infty,1}$ coincides with the so-called Sj\"ostrand's class \cite{lerner,sjostrand}.
Now, the inclusion relations with Besov spaces are well understood, see e.g. \cite{baoxiang}. The key estimates here are however the inclusion relations with the Sobolev spaces \cite[Theorem 1.3]{ks}, namely
\begin{equation}\label{eq22}
L^p_s\hookrightarrow M^p\hookrightarrow L^p,\qquad 1<p\leq 2,\ s=2d\big(\frac{1}{p}-\frac{1}{2}\big).
\end{equation}
Hence, for a linear operator $A$ we have 
\[
\|A\|_{L^p_s\to L^p}\lesssim\|A\|_{M^p\to M^p}
\]
for $1<p\leq 2$ and $s$ as above. The interesting fact is that this gives optimal estimates in many cases.\par For example, for $e^{it\Delta}$ one can combine the estimate
\[
\|e^{it\Delta}\|_{M^p\to M^p}\lesssim (1+|t|)^{d|\frac{1}{p}-\frac{1}{2} |},\quad 1\leq p\leq\infty,
\]
see e.g. \cite[Proposition 6.6]{baoxiang}, with \eqref{eq22} and deduce that
\[
\|e^{it\Delta}\|_{L^p_s\to L^p}\lesssim (1+|t|)^{d|\frac{1}{p}-\frac{1}{2}|}
\]
for $1<p<\infty$, $s=2d\big|\frac{1}{p}-\frac{1}{2}\big|$ (the case $p>2$ follows by duality). \par
This is sharp, and of course known. Consider however the following class of Fourier integral operators $A$, generalizing the propagator $e^{i\Delta}$. Suppose that $A$ has the form
\begin{equation}\label{fi1}
Af(x)=(2\pi)^{-d}\int_{\rd} e^{i\Phi(x,\eta)}a(x,\eta)\widehat{f}(\eta)\, d\eta
\end{equation}
with a phase ${\Phi}\in C^\infty(\R^d\times \R^d)$, real-valued and satisfying
\begin{equation}\label{fi2}
|\partial^\alpha_z \Phi(z)|\leq C_\alpha\quad |\alpha|\geq 2,\ z\in\rdd
\end{equation}
as well as
\begin{equation}\label{fi3}
\Big| {\rm det}\,\Big( \frac{\partial^2 \Phi}{\partial x_j\partial \eta_j} \Big) \Big|\geq \delta>0\quad {\rm in}\ \rd\times\rd,
\end{equation}
and a symbol ${a}$ in the class $S^0_{0,0}$, i.e.
\begin{equation}\label{fi4}
|\partial^\alpha_z a(z)|\leq C_\alpha,\quad \alpha\in\bN^d,\ z\in\rdd. 
\end{equation}
Operators of this type arise as propagators for Schr\"odinger equations with an Hamiltonian having quadratic growth (see \cite{helffer}), in contrast with Fourier integral operators with a phase positively homogeneous of degree $1$ in $\eta$, which are instead applied in the study of hyperbolic problems. For the latter class, the problem of the local and global $L^p$ continuity is well understood \cite{cnrtrans,coriasco,ferreira,seeger}. Instead, we are not aware of similar results for operators $A$ of the above form, except for the $L^2$ results in \cite{asada,ruz-sug}. Now, it was proved in \cite{cnr} that, under the above assumption, $A$ and its adjoint are bounded on $M^p$, for every $1\leq p\leq\infty$. As a consequence, we have the following result. 
\begin{theorem} Let $A$ be a Fourier integral operator as in \eqref{fi1}--\eqref{fi4}. Then for every $1<p<\infty$, $s=2d\big|\frac{1}{p}-\frac{1}{2}\big|$, we have 
\[
\|Af\|_{L^p}\lesssim \|f\|_{L^p_{s}},\quad 1<p\leq 2,
\] 
\[
\|Af\|_{L^p_{-s}}\lesssim \|f\|_{L^p},\quad 2\leq p<\infty.
\]
\end{theorem}
Similarly, one can consider Schr\"odinger operators with rough Hamiltonians. We consider here a simple case, to avoid technicalities. 
\begin{theorem} 
Consider a potential $V(x)$ satisfying $\partial^\alpha V\in M^{\infty,1}$, $|\alpha|=2$. Let $H=-\Delta+V(x)$, $1<p<\infty$, $s=2d\big|\frac{1}{p}-\frac{1}{2}\big|$. Then 
\[
\|e^{iH} f\|_{L^p}\lesssim \|f\|_{L^p_s}, \quad 1<p\leq 2,
\]
\[
\|e^{iH} f\|_{L^p_{-s}}\lesssim \|f\|_{L^p}, \quad 2\leq p< \infty.
\]
\end{theorem}
Indeed, it was proved in \cite{cnr-rough} that the propagator $e^{itH}$ is bounded in $M^p$, $1\leq p\leq\infty$, $t\in\R$. Much more general Hamiltonians can be treated similarly; we refer to \cite{cnr-rough} and the bibliography therein for more details.  
  
\end{document}